\documentclass[12pt,a4paper]{amsart}
\usepackage{amsmath}
\usepackage{amssymb}
\usepackage[english]{babel}
\usepackage{verbatim}
\usepackage[shortlabels]{enumitem}
\usepackage{color}
\usepackage{tikz-cd}

\newtheoremstyle{mio}%
	{}{} % spazio sopra e sotto%
	{\itshape}{} % corpo del testo, indentation
	{\bfseries}{.}{ } % titolo del teorema: tipo di testo, divisore, spaziatura
	{#1 #2\thmnote{~\mdseries(#3)}} % formattazione nota
\theoremstyle{mio}
\newtheorem{teor}{Theorem}[section]
\newtheorem{cor}[teor]{Corollary}
\newtheorem{prop}[teor]{Proposition}
\newtheorem{lemma}[teor]{Lemma}
\newtheorem{defin}[teor]{Definition}

\newtheoremstyle{definition2}%
	{}{} % spazio sopra e sotto%
	{}{} % corpo del testo, indentation
	{\bfseries}{.}{ } % titolo del teorema: tipo di testo, divisore, spaziatura
	{#1 #2\thmnote{\mdseries~ #3}} % formattazione nota
\theoremstyle{definition2}

\newtheorem{oss}[teor]{Remark}

\newcommand{\ins}[1]{\mathbb{#1}}
\newcommand{\insN}{\ins{N}}
\newcommand{\inN}{\in\insN}
\newcommand{\insZ}{\ins{Z}}
\newcommand{\insQ}{\ins{Q}}
\newcommand{\insC}{\ins{C}}
\newcommand{\Max}{\mathrm{Max}}
\newcommand{\Spec}{\mathrm{Spec}}
\newcommand{\inv}[1]{\frac{1}{#1}}

\newcommand{\nz}{\bullet}

\newcommand{\car}{\mathrm{char}}

\newcommand{\omef}{\mathrm{Homeo}}

\newcommand{\pow}[1]{\mathrm{pow}({#1})}
\newcommand{\limiti}{\mathcal{L}}

\title{The Golomb topology of polynomial rings}
\author{Dario Spirito}
\date{\today}
\address{Dipartimento di Matematica e Fisica, Universit\`a degli Studi ``Roma Tre'', Roma, Italy}
\email{spirito@mat.uniroma3.it}

\begin{document}

\begin{abstract}
We study properties of the Golomb topology on polynomial rings over fields, in particular trying to determine conditions under which two such spaces are not homeomorphic. We show that if $K$ is an algebraic extension of a finite field and $K'$ is a field of the same characteristic, then the Golomb spaces of $K[X]$ and $K'[X]$ are homeomorphic if and only if $K$ and $K'$ are isomorphic.
\end{abstract}

\maketitle

\section{Introduction}
Let $R$ be an integral domain. The \emph{Golomb space} of $R$ is the topological space $G(R)$ having $R^\nz:=R\setminus\{0\}$ as a base set, and whose topology is generated by the coprime cosets. This topology, introduced by Brown \cite{brown-golomb} on $\insZ^+$ and later studied by Golomb \cite{golomb-connectedtop,golomb-aritmtop}, is one of many coset topologies \cite{knopf}, and it can be used to generalize Furstenberg's ``topological'' proof of the infinitude of primes \cite{furstenberg,clark-euclidean}. 

Recently two papers, the first one by Banakh, Mioduszewski and Turek \cite{bmt-golomb} and the second one by Clark, Lebowitz-Lockard and Pollack \cite{clark-golomb}, have started studying more deeply the topology on $G(R)$ and the continuous maps between these spaces, with the former concentrating on the ``classical'' case of $\insZ^+$ and the latter generalizing several results to integral domains and, in particular, to Dedekind domains. A central problem of both is the \emph{isomorphism problem}: if $G(R)$ and $G(S)$ are homeomorphic topological spaces, must $R$ and $S$ be isomorphic rings? More generally, how much the continuous maps (and, in particular, homeomorphisms and self-homeomorphisms) of Golomb spaces respect the algebraic structure of the underlying rings? In \cite{golomb-almcyc}, it was shown that the unique self-homeomorphisms of $h:G(\insZ)\longrightarrow G(\insZ)$ are the identity or the multiplication by $-1$; the proof of this result relies crucially on the fact that the groups of units of the quotients $\insZ/p^n\insZ$ (where $p$ is a prime number) are very close to being cyclic.% -- more precisely, they become cyclic when quotiented by the image of $\{\pm 1\}$ of the group of units of $R$.

In this paper, we study the isomorphism problem in the context of polynomial rings over fields; in particular, we are interested in the more restricted problem of determining if the existence of a homeomorphism between $G(K[X])$ and $G(K'[X])$ implies that $K$ and $K'$ are isomorphic as fields. To do so, we study the closure of several sets under the Golomb topology and under the $P$-adic topologies (which can be reconstructed from the Golomb topology), obtaining several results that allow to determine algebraic properties of $K$ from the topological properties of $G(K[X])$. While we aren't able to solve the isomorphism problem in full generality, we show that if $K$ is an algebraic extension of a finite field, $K'$ has the same characteristic of $K$ and $G(K[X])\simeq G(K'[X])$ then $K'$ must be isomorphic to $K$ (Theorem \ref{teor:algextFp}); in particular, this implies that the number of distinct Golomb topologies associated to countable domains it the cardinality of the continuum, answering a question posed in \cite[Section 3.1]{clark-golomb}.

\medskip

The structure of the paper is as follows. In Section \ref{sect:prelim}, we fix the notation and recall some results that will be used throughout the paper. In Section \ref{sect:GnR} we give a few results on the set of units of $R$ under the Golomb topology. In Section \ref{sect:char0} we show that the Golomb topology allow to distinguish between zero and positive characteristic (Proposition \ref{prop:caratt0-ptop}). In Section \ref{sect:charp-sepclos} we study the case of separably closed fields in positive characteristic: we show that we can distinguish them from the other fields (Proposition \ref{prop:sepchiusi}) and that we can recover the characteristic of $K$ from $G(K[X])$ (Theorem \ref{teor:algchiusi-char}). In Section \ref{sect:G1} we generalize a result of \cite{bmt-golomb} on the image of prime elements under a homeomorphism; this is used in the following Section \ref{sect:algFp} to link a (topologically distinguished) subgroup of self-homeomorphisms of $G(K[X])$ with the unit group of $K$ (Proposition \ref{prop:Homeofix}) and thus to prove the aforementioned main theorem (Theorem \ref{teor:algextFp}).

\section{Preliminaries and notation}\label{sect:prelim}
Let $R$ be an integral domain; we shall always suppose that $R$ is not a field. Given a set $I\subseteq R$, we set $I^\nz:=I\setminus\{0\}$. We denote by $U(R)$ the set of units of $R$ (both as a set and as a group).

The \emph{Golomb space} of $R$ is the topological space having $R^\nz$ as base set and whose topology is generated by the coprime cosets of $R$, that is, by the sets $x+I$ where $x\in R^\nz$, $I$ is a nonzero ideal of $R$ and $\langle x,I\rangle=R$. We denote by $G(R)$ the Golomb space of $R$, and call the topology the \emph{Golomb topology} on $R$. When $R$ is an integral domain with zero Jacobson radical,\footnote{The \emph{Jacobson radical} of $R$ is the intersection of the maximal ideals of $R$.} $G(R)$ is a Hausdorff space that is not regular; furthermore, $G(R)$ is not compact, and is a connected space that is totally disconnected at each of its points \cite[Theorems 5, 8 and 9 and Proposition 10]{clark-golomb}.

Suppose from now on that $R$ is a Dedekind domain.

Given a subset $A\subseteq R^\nz$, we denote by $\overline{A}$ the closure of $A$ in the Golomb topology. Let $x+I$ be a coprime coset. If $I=P_1^{e_1}\cdots P_n^{e_n}$ is the factorization of $I$ into prime ideals, then \cite[Lemma 15]{clark-golomb}
\begin{equation*}
\overline{x+I}=\bigcap_{i=1}^n(P_i^\nz\cup(x+P_i^{e_i})).
\end{equation*}
If $h:G(R)\longrightarrow G(S)$ is a homeomorphism, then $h$ sends units into units (i.e., $h(U(R))=U(S)$) \cite[Theorem 13]{clark-golomb}. If the class group of $R$ is torsion then $h$ sends prime ideals into prime ideals, that is, if $P$ is a prime ideal of $R$ then $h(P^\nz)\cup\{0\}$ is a prime ideal of $S$; more generally, $h$ takes radical ideals into radical ideals \cite[Theorem 2.8]{golomb-almcyc}.

If $\Delta$ is a subset of $\Max(R)$, we denote by $G_\Delta(R)$ the set of all $x\in R$ that are contained into the elements of $\Delta$ and in no other prime ideal. If the class group of $R$ is torsion, this set is again preserved by homeomorphisms: if $h$ is a homeomorphism and $x\in G_\Delta(R)$, then $h(x)\in G_\Lambda(R)$, where $\Lambda$ contains the images under $h$ of the elements of $\Delta$ \cite[Proposition 2.7]{golomb-almcyc}. Given $a\in R$, we set $\pow{a}:=\{ua^n\mid u\in U(R),n\geq 1\}$; if $a$ generates $P$, then $\pow{a}$ is exactly $G_{\{P\}}(R)$.

Let now $R$ be a Dedekind domain with torsion class group and $P$ be a prime ideal of $R$. The \emph{$P$-topology} on $R\setminus P$ is the topology generated by the sets $a+P^n$, for all $a\in R\setminus P$ and all $n\geq 1$; this is exactly the restriction of the $P$-adic topology on $R\setminus P$. The $P$-topology can be recovered from the Golomb topology by considering only the clopen subset of $R\setminus P$, and thus every homeomorphism $h:G(R)\longrightarrow G(S)$ in the Golomb topology restricts to a homeomorphism between $R\setminus P$ and $S\setminus Q$ (with $Q:=h(P^\nz)\cup\{0\}$), where the former is endowed with the $P$-topology and the latter with the $Q$-topology \cite[Section 3]{golomb-almcyc}.

We denote by $\car K$ the characteristic of the field $K$. If $q$ is a prime power, we denote by $\ins{F}_q$ the finite field with $q$ elements. If $p$ is a prime number, we denote by $\ins{F}_{p^\infty}$ the algebraic closure of $\ins{F}_p$.

\section{The spaces $G_n(R)$}\label{sect:GnR}
Let $R$ be an integral domain. We denote by $G_0(R)$ the set of units of $R$ endowed with the Golomb topology; this space is rather more well-behaved than the whole Golomb space.
\begin{prop}\label{prop:G0R}
Let $R$ be an integral domain.
\begin{enumerate}[(a)]
\item $G_0(R)$ is homogeneous.
\item Suppose the Jacobson radical of $R$ is zero. Then, $G_0(R)$ is discrete if and only if there is an ideal $I$ such that the restriction $G_0(R)\longrightarrow R/I$ of the canonical quotient is injective.
\item $G_0(R[X])$ is discrete.
\end{enumerate}
\end{prop}
\begin{proof}
Since multiplication by units is a homeomorphism, we can always send $x$ to $y$ by multiplying by $yx^{-1}$; hence $G_0(R)$ is homogeneous.

For the second claim, suppose first that $G_0(R)\longrightarrow R/I$ is injective: then, for every unit $u$ the coset $u+I$ meets $G_0(R)$ only in $u$, and thus $G_0(R)$ is discrete. Conversely, suppose $G_0(R)$ is discrete: then, there is an ideal $I$ such that $(1+I)\cap G_0(R)=\{1\}$. For every other unit $u$ of $R$, $u+I=u(1+I)$; hence, $u$ is the only unit in $(u+I)\cap G_0(R)$. Thus, $G_0(R)\longrightarrow R/I$ is injective.

The last claim follows taking $I=XR[X]$.
\end{proof}

When $R$ is a Dedekind domain we can say more.
\begin{prop}\label{prop:G0-ded}
Let $R$ be a Dedekind domain with zero Jacobson radical. 
\begin{enumerate}[(a)]
\item\label{prop:G0-ded:clopen} $G_0(R)$ has a basis of clopen sets.
\item\label{prop:G0-ded:reg} $G_0(R)$ is regular.
\item\label{prop:G0-ded:count} If $R$ is countable, then $G_0(R)$ is either discrete or homeomorphic to $\insQ$ (endowed with the Euclidean topology).
\item\label{prop:G0-ded:residue} If $R$ is countable and every residue field of $R$ is finite, then $G_0(R)\simeq\insQ$.
\end{enumerate}
\end{prop}
\begin{proof}
\ref{prop:G0-ded:clopen} We need to show that $(x+I)\cap G_0(R)$ is clopen in $G_0(R)$ for every $x\in G_0(R)$ and every ideal $I$. Indeed, let $I=\prod_iP_i^{e_i}$ be the factorization of $I$; then, by \cite[Lemma 15]{clark-golomb},
\begin{equation*}
\overline{x+I}\cap G_0(R)=\bigcap_i(P_i^\nz\cup(x+P_i^{e_i}))\cap G_0(R).
\end{equation*}
Since $P_i^\nz\cap G_0(R)=\emptyset$, we have $\overline{x+I}\cap G_0(R)=\bigcap_i((x+P_i^{e_i})\cap G_0(R))=(x+I)\cap G_0(R)$, i.e., $(x+I)\cap G_0(R)$ is clopen in $G_0(R)$.

\ref{prop:G0-ded:reg} Let $x\in G_0(R)$ and $V\subseteq G_0(R)$ be a closed set not containing $x$; then, $G_0(R)\setminus V$ is open, and thus it contains a basic clopen set $(x+I)\cap G_0(R)$. Therefore, $x$ and $V$ are separated by $(x+I)\cap G_0(R)$ and $G_0(R)\setminus(x+I)$, and so $G_0(R)$ is regular.

\ref{prop:G0-ded:count} If $R$ is countable, then it has only countably many ideals, and thus $R$ and $G_0(R)$ are second countable. Thus, it is metrizable \cite[e-2]{encytop}. If $G_0(R)$ is not discrete, it follows that $G_0(R)\simeq\insQ$ \cite{sierpinski-Q,dagupta-atlas}. Finally, \ref{prop:G0-ded:residue} now follows from this and Proposition \ref{prop:G0R}.
\end{proof}

We now introduce a sequence $\{G_n(R)\}_{n\inN}$ of subspaces of $G(R)$ generalizing $G_0(R)$.
\begin{defin}
Let $R$ be a Dedekind domain. For every $n\geq 0$, define
\begin{equation*}
G_n(R):=\bigcup\{G_\Delta(R)\mid \Delta\subseteq\Max(R),|\Delta|=n\}.
\end{equation*}
\end{defin}
By \cite[Proposition 2.7]{golomb-almcyc}, if $R$ has torsion class group then the topology of the $G_n(R)$ is uniquely determined by the Golomb topology, in the sense that if $h:G(R)\longrightarrow G(S)$ is a homeomorphism then $h(G_n(R))=G_n(S)$ and thus $G_n(R)$ and $G_n(S)$ are homeomorphic.

The results proved above for $n=0$ does not generalize to arbitrary $n$. When $n=1$, we can prove a partial analogue of Proposition \ref{prop:G0-ded}\ref{prop:G0-ded:reg} by extending the proof of \cite[Theorem 3.1]{bmt-golomb}.
\begin{prop}
Let $R$ be a Dedekind domain that is not a field, and suppose that $R$ has finitely many units. Then, $G_1(R)$ is a regular space.
\end{prop}
\begin{proof}
Let $\Omega$ be an open set of $G(R)$ and let $x\in G_1(R)\cap\Omega$; we need to show that there is an open neighborhood $O$ of $x$ such that $\overline{O}\cap G_1(R)\subseteq\Omega\cap G_1(R)$. Without loss of generality, we can suppose that $\Omega=x+bR$ for some $b$ coprime with $x$.

Let $P_1,\ldots,P_n$ be the prime ideals containing $b$; then, the set $\Lambda$ of the prime elements contained in some $P_i$ is finite (as $R$ has finitely many units). Thus, the set $x-\Lambda:=\{x-p\mid p\in\Lambda\}$ is finite too, and so there are only finitely many prime ideals that contain some element of $x-\Lambda$. 

Since $R$ has finitely many units, it has infinitely many maximal ideals; thus, there is a prime ideal $Q$ that is distinct from each $P_i$ and that do not contain $x$ nor any element of $x-\Lambda$. Consider $O:=x+bQ$: then, $O$ is a coprime coset, and thus it is open. By \cite[Lemma 15]{clark-golomb},
\begin{equation*}
\overline{O}=\bigcap_i(P_i^\nz\cup(x+P_i^{e_i}))\cap(Q^\nz\cup(x+Q)),
\end{equation*}
where $e_i$ is the exponent of $P_i$ in the factorization of $bR$.

Let $p\in\overline{O}\cap G_1(R)$. By definition, $p$ can be contained in at most one of $P_1,\ldots,P_n,Q$. We distinguish three cases.
\begin{itemize}
\item If $p$ is not contained in any of them, then $p\in\bigcap_i(x+P_i^{e_i})\cap(x+Q)=(x+bR)\cap(x+Q)=x+bQ=O\subseteq\Omega$. 
\item If $p$ is contained in $P_i$ for some $i$, then it should be contained in $x+Q$, that is, $p-x\in Q$. However, this contradicts the choice of $Q$.
\item If $p\in Q$, then we must have $p\in\bigcap_i(x+P_i^{e_i})=x+bR=\Omega$.
\end{itemize} 
Hence, $\overline{O}\cap G_1(R)\subseteq\Omega\cap G_1(R)$, as claimed. Thus, $G_1(R)$ is regular.
\end{proof}

Like for $G_0(R)$, this implies that if $R$ is also countable then $G_1(R)$ is second countable and thus metrizable; hence, it is either discrete or homeomorphic to $\insQ$. Another important property that $G_1(R)$ may or may not have is being dense in $G(R)$; we shall study more deeply this aspect in Section \ref{sect:algFp}.

For $n\geq 2$, however, the space $G_n(R)$ is usually not regular. We need a topological lemma.
\begin{lemma}\label{lemma:clos-dense}
Let $X$ be a topological space, $Y\subseteq X$ a dense subset and $\Omega$ an open subset of $X$. Then, $\overline{\Omega}\cap Y=\overline{\Omega\cap Y}\cap Y$.
\end{lemma}
\begin{proof}
Clearly, $\overline{\Omega\cap Y}\cap Y\subseteq\overline{\Omega}\cap Y$. On the other hand, let $x\in\overline{\Omega}\cap Y$. If $x\notin\overline{\Omega\cap Y}$, then there is an open set $O$ of $X$ containing $x$ but disjoint from $\Omega\cap Y$, that is, $O\cap\Omega\cap Y=\emptyset$. However, since $Y$ is dense and $O\cap\Omega$ is open it follows that $O\cap\Omega=\emptyset$, and thus $x\notin\overline{\Omega}$, a contradiction. It follows that $\overline{\Omega}\cap Y\subseteq\overline{\Omega\cap Y}\cap Y$. The claim is proved.
\end{proof}

\begin{prop}\label{prop:Gn>1}
Let $R$ be a Dedekind domain with torsion class group such that $G_1(R)$ is dense in $G(R)$. Then, for every $n\geq 2$,
\begin{enumerate}[(a)]
\item\label{prop:Gn>1:dense} $G_n(R)$ is dense in $G(R)$;
\item\label{prop:Gn>1:reg} $G_n(R)$ is not regular.
\end{enumerate}
\end{prop}
\begin{proof}
\ref{prop:Gn>1:dense} If $x+bR$ is a subbasic open set, we can find $p_1\in(x+bR)\cap(1+xR)\cap G_1(R)$; then, as $p_1$ is coprime with $x$, the set $x+p_1bR$ is open, and thus we can find $p_2\in (1+p_1bR)\cap G_1(R)$, then $p_3\in(1+p_1p_2bR)\cap G_1(R)$, and so on. Then, $c:=p_1\cdots p_n$ will be an element of $G_n(R)$ (as each $p_i$ is in $G_1(R)$ and $p_i$ and $p_j$ are coprime for $i\neq j$) such that $c\equiv x\cdot 1\cdots 1=x\bmod bR$, i.e., $c\in(x+bR)\cap G_n(R)$. Hence, $G_n(R)$ is dense.

\ref{prop:Gn>1:reg} Fix $n\geq 2$, and let $p\in G_1(R)$. Let $\Omega:=1+pR$, take $x\in\Omega\cap G_n(R)$, and let $O$ be an open set such that $x\in O$ and $O\cap G_n(R)\subseteq\Omega\cap G_n(R)$. We claim that $\overline{O}\cap G_n(R)\nsubseteq\Omega$. Without loss of generality we can take $O=x+bR$, with $b$ coprime to $x$; furthermore, passing if needed to a power $b^k$ we can suppose that $b$ is a product of primary elements.

If $x+b\in pR$, then we can write $x+b=py$ for some $y\in R$, and $py+pbR\subseteq O$ since $x+b+pbR\subseteq x+bR$. Let $O':=y+bR$; then, $O'$ is open (if $y$ and $b$ have a common factor $s$, then $s$ would divide also $x$, a contradiction). Since $G_{n-1}(R)$ is dense, we can find $q\in O'\cap G_{n-1}(R)$; then, $pq\in O\cap G_n(R)$, while $pq\notin\Omega$ as $pq\in pR$. This contradicts $O\cap G_n(R)\subseteq\Omega\cap G_n(R)$.

Therefore, $x+b\notin pR$. Let $b:=b_1\cdots b_n$, where each $b_i$ belongs to $G_1(R)$ and $b_i$ and $b_j$ are coprime if $i\neq j$. If $b_i\in pR$ for some $i$, let $b':=b/b_i$; otherwise, set $b':=b$. Then, $p$ is coprime with $b'$, and thus there is a $z\in R$, coprime with $p$, such that $pz\equiv x\bmod b'R$. By density, there is a $q\in(z+bR)\cap G_{n-1}(R)$; we claim that $pq\in(\overline{O}\cap G_n(R))\setminus\Omega$. Indeed, it is clear that $pq\in G_n(R)$ (since $p\in G_1(R)$, $q\in G_{n-1}(R)$ and $p$ and $q$ are coprime), and $pq\notin\Omega$ since $pq\in pR$.  By \cite[Lemma 15]{clark-golomb},
\begin{equation*}
\overline{O}=\bigcap_i(P_i^\nz\cup(x+b_iR)),
\end{equation*}
where $P_i$ is the prime ideal containing $b_i$. If $b_i$ is not coprime with $p$, then $pq\in P_i^\nz\subseteq\overline{O}$. If $b_i$ is coprime with $p$, then $b_i$ divides $b'$ and
\begin{equation*}
pq\in p(z+bR)=pz+pbR\subseteq pz+b'R=x+b'R\subseteq x+b_iR\subseteq\overline{O}.
\end{equation*}
Hence, $pq\in(\overline{O}\cap G_n(R))\setminus\Omega$.

Let $V:=G_n(R)\setminus\Omega$: then, $V$ is a closed set of $G_n(R)$. If $G_n(R)$ were regular, then there would be disjoint open sets $O_1,O_2$ such that $x\in O_1\cap G_n(R)$ and $V\subseteq O_2\cap G_n(R)$. In particular, $O_1\cap G_n(R)\subseteq G_n(R)\setminus(O_2\cap G_n(R))$, and the latter is a closed set; therefore, the closure $V'$ of $O_1\cap G_n(R)$ inside $G_n(R)$ would be disjoint from $V$. However, by Lemma \ref{lemma:clos-dense},
\begin{equation*}
V'=\overline{O_1\cap G_n(R)}\cap G_n(R)=\overline{O_1}\cap G_n(R);
\end{equation*}
by the previous part of the proof, $\overline{O_1}\cap G_n(R)$ is not contained in $\Omega$, i.e., it meets $V$. This is a contradiction, and thus $G_n(R)$ is not regular.
\end{proof}

\section{Characteristic 0}\label{sect:char0}

We now start studying the Golomb spaces $G(K[X])$ of polynomial rings over fields. In this section, we analyze what happens when the characteristic of the field is $0$. The first result is that we can actually distinguish them from the positive characteristic case.
\begin{prop}\label{prop:caratt0-ptop}
Let $K$ be a field. Then, $K$ has characteristic $0$ if and only if there is an irreducible polynomial $g\in K[X]$ such that $\pow{g}$ is closed in the $P$-topology for every prime ideal $P$ not containing $g$.
\end{prop}
\begin{proof}
Suppose $K$ has characteristic $0$, and choose $g(X):=X$. Let $P=(f)$ be a prime ideal not containing $g$, and let $\lambda\notin(P\cup\pow{g})$: suppose that $\lambda$ is in the closure of $\pow{g}$ in the $P$-topology. Then, for every $n\inN^+$ the open set $\lambda+P^n$ contains an element of $\pow{g}$. Take $n>\deg\lambda+1$: then, there are $m\inN^+$ and $u\in K^\nz$ such that $ug^m\in\lambda+P^n$, i.e., $f^n$ divides $h:=\lambda-ug^m$. Since $\lambda\notin\pow{g}$, $h\neq 0$, and thus $m\geq n$. Let $H$ the $(\deg\lambda+1)$-th formal derivative of $h$: then, $\lambda$ goes to $0$, and thus $H$ is the $(\deg\lambda+1)$-th formal derivative of $-ug^m=-uX^m$, that is, $H(X)=cX^{m-\deg\lambda-1}$ for some $c\in K$. Since $\car K=0$ and $m>\deg\lambda+1$, we have $H\neq 0$, and thus its unique zero is $0$. This contradicts the facts that $f|H$ and that $f$ is coprime with $X$. Hence, $\pow{g}$ is closed in the $(f)$-topology.

Conversely, suppose there is a polynomial $g$ with this property, and suppose that $\car K=p>0$. Let $a\in K$ be such that $g(a)\neq 0$ (which exists since $g$ is irreducible). Then, $f(X):=X-a$ divides $1-\frac{g(X)}{g(a)}$, and thus $f^{p^n}$ divides $\left(1-\frac{g(X)}{g(a)}\right)^{p^n}=1-\frac{g(X)^{p^n}}{g(a)^{p^n}}$, that is, $1+(f)^{p^n}$ meets $\pow{g}$. Therefore, $1+(f)^k$ meets $\pow{g}$ for every $k\inN^+$, i.e., $1$ is in the closure of $\pow{g}$ in the $(f)$-topology. This contradicts the choice of $g$, and thus the characteristic of $K$ must be $0$, as claimed.
\end{proof}

\begin{cor}\label{cor:car0p}
Let $K_1,K_2$ be fields. If $\car K_1=0$ and $\car K_2>0$, then the Golomb spaces $G(K_1[X])$ and $G(K_2[X])$ are not homeomorphic.
\end{cor}
\begin{proof}
If $g$ is an irreducible polynomial of $K[X]$, then $\pow{g}=G_{\{(g)\}}(K[X])$. By the previous proposition, $\car K=0$ if and only if there is a prime ideal $P$ such that $G_{\{P\}}(K[X])$ is closed in the $Q$-topology for every prime ideal $Q\neq P$. Since any homeomorphism of Golomb spaces sends prime ideals into prime ideals, this property is preserved by homeomorphisms. In particular, if $G(K_1[X])\simeq G(K_2[X])$ then $\car K_1=0$ if and only if $\car K_2=0$.
\end{proof}

Note that the proof of Proposition \ref{prop:caratt0-ptop} is qualitative, and thus cannot be readily applied to distinguish different positive characteristics. We shall do this in the algebraically closed case in Theorem \ref{teor:algchiusi-char}.

We now study the algebraically closed and the real closed case.
\begin{prop}\label{prop:C-GnK}
Let $K$ be an algebraically closed field of characteristic $0$. For every $n\geq 0$, $G_n(K[X])$ is discrete and closed in $G(K[X])$.
\end{prop}
\begin{proof}
Let $p(X)\in K[X]$, and let $b\in K$ be such that $p(b)\neq 0$ (which exists since $K$ is infinite). We claim that, for large $N$, the only possible element of $(p(X)+(X-b)^NK[X])\cap G_n(K[X])$ is $p(X)$.

Indeed, suppose that $q(X)\in(p(X)+(X-b)^NK[X])\cap G_n(K[X])$ is different from $p(X)$: then, we have
\begin{equation*}
\begin{cases}
q(X)=p(X)+(X-b)^Na(X)\\
q(X)=u(X-a_1)^{e_1}\cdots(X-a_n)^{e_n},
\end{cases}
\end{equation*}
where $a(X)\neq 0$, $a_1,\ldots,a_n$ are distinct, $e_1,\ldots,e_n\geq 1$ and $u\in K$. Let $d:=\deg p$, and  apply $d+1$ times the formal derivative process. In the first equation, $p^{(d+1)}$ becomes $0$, and thus (since $a(X)\neq 0$) $q^{(d+1)}$ has a zero of multiplicity $N-d-1$ in $b$. In the second equation, at each step the multiplicity of each $a_i$ is lowered by $1$, and thus each $a_i$ is a zero of multiplicity at least $e_i-d-1$ (this holds even if $e_i<d+1$). Since $p(X)$ and $X-b$ are coprime, it follows that $b\neq a_i$ for each $i$; hence, the total multiplicities of the zeros of $q^{(d+1)}$ is at least
\begin{equation*}
N-d-1+\sum_i(e_i-d-1)=N+\sum_ie_i-(n+1)(d+1)=N+\deg q-(n+1)(d+1).
\end{equation*}
Both $n$ and $d$ are fixed; hence, choosing $N>n(d+1)$, we have (using the fact that $K$ has characteristic $0$)
\begin{equation*}
\deg q^{(d+1)}>n(d+1)+\deg q-(n+1)(d+1)=\deg q-(d+1)=\deg q^{(d+1)},
\end{equation*}
a contradiction. Hence, $(p(X)+(X-b)^NK[X])\cap G_n(K[X])$ contains at most $p(X)$. 

Therefore, if $p(X)\notin G_n(K[X])$ then $p(X)+(X-b)^NK[X]$ is disjoint from $G_n(K[X])$, and thus $p(X)$ is not in the closure of $G_n(K[X])$; on the other hand, if $p(X)\in G_n(K[X])$ then $(p(X)+(X-b)^NK[X])\cap G_1(K[X])=\{p(X)\}$ is an open set of $G_n(K[X])$. Hence, $G_n(K[X])$ is discrete and closed in $G(K[X])$, as claimed.
\end{proof}

\begin{cor}
Let $K$ be a real closed field. For every $n\geq 0$, $G_n(K[X])$ is discrete and closed in $G(K[X])$.
\end{cor}
\begin{proof}
Let $K'$ be the algebraic closure of $K$. Then, the inclusion map $G(K[X])\longrightarrow G(K'[X])$ sends $G_n(K[X])$ inside $G':=G_n(K'[X])\cup\cdots\cup G_{2n}(K'[X])$; the latter set is discrete, and thus so is $G_n(K[X])$. Furthermore, $G_n(K[X])$ is closed in $G'$, and thus $G_n(K[X])$ is also closed in $G(K'[X])$; in particular, $G_n(K[X])$ is closed in $G(K[X])$.
\end{proof}

These results can be used, for example, to distinguish $G(\insQ[X])$ from $G(\overline{\insQ}[X])$. See Section \ref{sect:G1}.

\section{Separably closed fields in characteristic $p$}\label{sect:charp-sepclos}
In this section, we analyze what happens to fields of positive characteristic that are separably or algebraically closed. The first step is distinguishing them from the other fields; the following proof is similar to the proof of Proposition \ref{prop:caratt0-ptop}.
\begin{prop}\label{prop:sepchiusi}
Let $K$ be a field of characteristic $p>0$, and suppose that $K$ is transcendental over $\ins{F}_p$. Then, $K$ is separably closed if and only if, for every coprime irreducible polynomials $f,g$ of $K[X]$, $G_0(R)$ is contained in the closure of $\pow{g}$ in the $(f)$-topology.
\end{prop}
\begin{proof}
Suppose first that $K$ is separably closed; since $\pow{g}$ is invariant under multiplication by units, it is enough to show that $1$ is in the closure of $\pow{g}$. Write $f(X)=X^{p^n}-a$, and let $\alpha$ be a root of $f$ in the algebraic closure $\overline{K}$ of $K$. Then, $h:=1-\inv{g(\alpha)}g$ is a polynomial over $\overline{K}$ having $\alpha$ as a zero, and thus $X-\alpha$ divides $h$; hence, $f(X)=(X-\alpha)^{p^n}$ divides
\begin{equation*}
h^{p^n}=\left(1-\inv{g(\alpha)}g\right)^{p^n}=1-\inv{g(\alpha)^{p^n}}g^{p^n}
\end{equation*}
in $\overline{K}[X]$. However, $g(\alpha)^{p^n}\in K[X]$, and thus $f$ divides $h^{p^n}$ also in $K[X]$. Therefore, for every power $q$ of $p$, $f^q$ divides $(h^{p^n})^q=1-\inv{g(\alpha)^{qp^n}}g^{qp^n}$, and in particular $1+f^qK[X]$ contains an element of $\pow{q}$. Thus, $1$ is in the closure of $\pow{q}$ under the $(f)$-topology, as claimed.

Conversely, suppose that $K$ is not separably closed, let $f$ be a separable irreducible polynomial, and let $\alpha,\beta$ be two distinct roots of $f$ in the algebraic closure of $K$; since $K$ is transcendental over $\ins{F}_p$, we can suppose that $\alpha,\beta$ are transcendental too. We claim that there is a $t\in K\cap\ins{F}_{p^{\infty}}$ such that $1$ is not in the closure of $\pow{X-t}$ in the $(f)$-topology. Indeed, suppose $1$ is in the closure for some $t$. Then, $\pow{X-t}$ meets $1+fK[X]$, and in particular there are a unit $u$ and an integer $m$ such that $f$ divides $1-u(X-t)^m$. Hence, we must have $1-u(\alpha-t)^m=0=1-u(\beta-t)^m$, and thus $(\alpha-t)/(\beta-t)$ must be a root of unity (of degree at most $m$), and in particular it must be algebraic over $\ins{F}_p$.

Let $r(t):=(\alpha-t)/(\beta-t)$ and $r:=r(0)=\alpha/\beta$. Then,
\begin{equation*}
r(t)=\frac{\alpha-t}{\beta-t}=\frac{r\beta-t}{\beta-t}\Longrightarrow\beta=\frac{t(r(t)-1)}{r(t)-r}
\end{equation*}
whenever $t\neq 0$ (which implies $r(t)\neq r$). However, both $t$ and $r(t)$ are algebraic over $\ins{F}_p$, and thus $\beta$ should be algebraic too; this is a contradiction, and thus $1$ is not in the closure of any $\pow{X-t}$ with $t\neq 0$.
\end{proof}

The following proposition shows the difference between the behavior of $G_n(K[X])$ in positive characteristic with respect to the characteristic 0 case (Proposition \ref{prop:C-GnK}). Part \ref{prop:Gn-pinfty:dense} does not hold without the assumption that $K$ is separably closed; indeed, its failure is critical in the proof of Proposition \ref{prop:Homeofix}.
\begin{prop}\label{prop:Gn-pinfty}
Let $K$ be a field of characteristic $p>0$. Then, the following hold.
\begin{enumerate}[(a)]
\item\label{prop:Gn-pinfty:dense} If $K$ is separably closed, then $G_1(K[X])$ is not dense in $G(K[X])$.
\item\label{prop:Gn-pinfty:isolated} If $K$ is algebraic over $\ins{F}_p$, then $G_n(K[X])$ has no isolated points for all $n\geq 1$.
\item\label{prop:Gn-pinfty:closure} If $K$ is algebraic over $\ins{F}_p$, then $G_m(K[X])$ is contained in the closure of $G_n(K[X])$ for all $n\geq m\geq 0$.
%\item\label{prop:Gn-pinfty:closed} For $n\geq 1$, $G_n(K[X])$ is not closed.
\end{enumerate}
\end{prop}
Note that all three of these hypothesis are fulfilled when $K$ is the algebraic closure of $\ins{F}_p$.
\begin{proof}
\ref{prop:Gn-pinfty:dense} Suppose first $p\geq 3$, and consider the open set $1+X^2+X^3K[X]$: if it intersects $G_1(K[X])$ then there are an irreducible polynomial $g(X)$, $u\in K$, $k\inN$ and $b(X)\in K[X]$ such that $ug(X)^k=1+X^2+X^3b(X)$. Since $K$ is separably closed, we can write $g(X)=X^{p^r}-a$ for some $r\geq 0$ and some $a\in K$. If $r>0$, then $ug(X)^k$ has no monomial of degree 2, a contradiction; hence, it must be $g(X)=X-a$. Considering the coefficients of degree 1 and 2, we have
\begin{equation*}
\begin{cases}
0=u\binom{k}{1}a^{k-1}=uk(-1)^{k-1}a^{k-1}\\
1=u\binom{k}{2}a^{k-2}=u\frac{k(k-1)}{2}(-1)^{k-2}a^{k-2}.
\end{cases}
\end{equation*}
The second equality implies that $k$, $k-1$ and $a$ are all different from $0$ in $K$; however, this implies $uka^{k-1}\neq 0$, a contradiction. Hence, $1+X^2$ does not belong to the closure of $G_1(K[X])$.

Suppose now $p=2$, and consider the open set $1+X+X^3+X^4K[X]$. Considering the monomial of degree 1, we see that the irreducible polynomial $g(X)$ must be in the form $X-a$. Suppose thus $1+X+X^3+X^4b(X)=u(X-a)^k$: then, confronting the coefficients of degree 1 we have that $k$ is odd, while confronting the coefficients of degree 2 we get that $(k-1)/2$ is even. The coefficient of degree 3 of $u(X-a)^k$ is thus $uk(k-2)\frac{k-1}{2}(-1)^{k-3}a^{k-3}=0$, contradicting the presence of $X^3$. Hence, $1+X+X^3$ does not belong to the closure of $G_1(K[X])$.

\ref{prop:Gn-pinfty:isolated} Let $a(X)\in G_n(K[X])$, and let $b(X)$ be a polynomial coprime with $a(X)$. Let $q$ be a prime power such that $\ins{F}_q$ contains all the coefficients of $a(X)$ and of $b(X)$. Then, $a(X)$ and $b(X)$ are coprime in $\ins{F}_q[X]$; since $\ins{F}_q[X]/b(X)\ins{F}_q[X]$ is finite, we can find $k>0$ such that $a(X)^k\equiv 1\bmod b(X)\ins{F}_q[X]$, and thus $a(X)^{k+1}\in G_n(K[X])\cap(a(X)+b(X)K[X])$ is different from $a(X)$. Hence, $a(X)$ is not isolated in $G_n(K[X])$.

\ref{prop:Gn-pinfty:closure} If $K$ is not algebraically closed then $G_1(K[X])$ is dense in $G(K[X])$ (see Proposition \ref{prop:dirichlet-aeff} below) and thus by Proposition \ref{prop:Gn>1}\ref{prop:Gn>1:dense} the $G_n(R)$ are actually dense.

Suppose that $K$ is algebraically closed: by part \ref{prop:Gn-pinfty:dense}, $G_0(R)$ is in the closure of $G_1(R)$.

Let now $a(X)\in G_m(K[X])$, with $m<n$, and let $b(X)$ be coprime with $a(X)$; let $r:=n-m$. Choose $r$ distinct elements, $t_1,\ldots,t_r$, such that $b(t_i)\neq 0$ and $a(t_i)\neq 0$ for all $i$; since $K[X]/b(X)K[X]$ is finite, we can find positive integers $k_1,\ldots,k_r$ such that $(X-t_i)^{k_i}\in 1+b(X)K[X]$ for all $i$. Let $A(X):=a(X)(X-t_1)^{k_1}\cdots(X-t_r)^{k_r}$: by construction, $A(X)\in G_n(K[X])$ and $A(X)\equiv a(X)\bmod b(X)K[X]$, that is, $A(X)\in a(X)+b(X)K[X]$. Hence, all neighborhood of $a(X)$ intersect $G_n(K[X])$, and thus $a(X)$ is in the closure of $G_n(K[X])$, as claimed.
%\ref{prop:Gn-pinfty:closed} follows immediately from the previous point, as $G_0(K[X])$ is disjoint from $G_n(K[X])$ but is in its closure.
\end{proof}

We now deal with the problem of distinguishing separably closed fields of different characteristics, that is, we want to prove that if $G(K[X])\simeq G(K'[X])$ then $K$ and $K'$ have the same characteristic, extending Proposition \ref{prop:caratt0-ptop}. Until the end of the section the section, $K$ will be a field of characteristic $p>0$ and $\overline{K}$ a (fixed) algebraic closure of $K$. We denote by $v_p$ the $p$-adic valuation on the positive integers.

\begin{defin}
Let $r(X)\in K[X]$ be an irreducible polynomial. A \emph{$r(X)$-sequence} is a sequence $E\subset \pow{r(X)}$. If $r(X)\notin(X-1)$, we say that $E$ is \emph{basic} if $E$ converges to $1$ in the $(X-1)$-topology.
\end{defin}

Since $E\subseteq\pow{q(X)}$, we can always write the elements of a $r(X)$-sequence $E:=\{s_n(X)\}_{n\inN}$ as  $s_n(X):=u_nr(X)^{e_n}$, for some $u_n\in K^\nz$ and some positive integers $e_n$.

\begin{lemma}\label{lemma:binomial}
Let $p$ be a prime number and $e,z$ be natural numbers such that $p^z<e$. If $p$ divides the binomial coefficient $\binom{e}{p^t}$ for all $1\leq t\leq z$, then $v_p(e)\geq z+1$.
\end{lemma}
\begin{proof}
Fixed $p$ and $e$, we proceed by induction on $z$. If $z=0$, then we know that $p$ divides $\binom{e}{p^0}=\binom{e}{1}=e$, and the claim holds.

Suppose we have proved the claim up to $z-1$. Then, $p^z|e$ and $p$ divides
\begin{equation*}
\binom{e}{p^z}=\frac{e(e-1)\cdots(e-p^z+1)}{p^z(p^z-1)\cdots 2\cdot 1}.
\end{equation*}
For all $0<k<p^z$, we have $v_p(k)<v_p(e)$ and thus $v_p(e-k)=\min\{v_p(e),v_p(k)\}=v_p(k)$; hence, the $p$-valuation of the product $(e-1)\cdots(e-p^z+1)$ is equal to the $p$-valuation of $(p^z-1)!$. Thus,
\begin{equation*}
0<v_p\left(\binom{e}{p^z}\right)=v_p\left(\frac{e}{p^z}\right)=v_p(e)-v_p(p^z)=v_p(e)-z.
\end{equation*}
It follows that $v_p(e)>z$, i.e., $v_p(e)\geq z+1$. By induction, the claim is proved.
\end{proof}

\begin{prop}\label{prop:convergenza-Xseq}
Let $r(X),q(X)$ be coprime irreducible polynomials, and let $E=\{s_n(X):=u_nr(X)^{e_n}\}_{n\inN}$ be a $r(X)$-sequence. Let $s\in K^\nz$. Then, $E$ converges to $s$ in the $(q(X))$-topology if and only if $v_p(e_n)\to\infty$ and $s_n(\lambda)$ is definitively equal to $s$ for every root $\lambda$ of $q(X)$ in $\overline{K}$.
\end{prop}
\begin{proof}
Suppose first that $K=\overline{K}$ is algebraically closed. Then, we can write $r(X):=X-t$, $q(X):=X-\lambda$ for some $t,\lambda\in K$. Let $Q:=(X-\lambda)$. 

Suppose the two conditions hold, and let $k$ be any integer. Then, there is an $N$ such that $v_p(e_n)\geq k$ and $s_n(\lambda)=s$ for every $n\geq N$. Thus,
\begin{align*}
s_n(\lambda) & =u_n(X-t)^{e_n}=u_n(X-\lambda+\lambda-t)^{p^ke'_n}=\\
& =u_n((X-\lambda)^{p^k}+(\lambda-t)^{p^k})^{e'_n}.
\end{align*}
Untying the binomial, we obtain $u_n((\lambda-t)^{p^k})^{e'_n}=u_n(\lambda-t)^{e_n}=s_n(\lambda)=s$, while the other monomials are all divisible by $(X-\lambda)^{p^k}$. Therefore, $s_n(\lambda)\in s+Q^{p^k}$ for all $n\geq N$. Since $\{s+Q^{p^k}\}$ is a local basis of neighborhoods of $s$ in the $Q$-topology, $E$ tends to $s$.

Conversely, if $E$ converges to $s$ in the $Q$-topology then $s_n(X)\in s+Q$ definitively, i.e., $s_n(X)-s\in Q$, or equivalently $q(X)$ divides $s_n(X)-s$. Hence, $s_n(\lambda)-s=0$ and $s_n(\lambda)=s$ definitively. We now have
\begin{equation*}
s_n(X)=u_n(X-\lambda+\lambda-t)^{e_n}=u_n\sum_i \binom{e_n}{i}(\lambda-t)^{n-i}(X-\lambda)^i.
\end{equation*}
Since $E$ converges to $s$, the polynomial $s_n(X)-s$ must belong (definitively) to $Q^k$ for every $k>0$, that is, the coefficients of degree $<k$ in $X-\lambda$ must be equal to $0$. Choose $k=p^z+1$. Then, for large $n$, we have that $\binom{e_n}{r}=0$ for all $1\leq r\leq p^z$; by Lemma \ref{lemma:binomial}, we have $v_p(e_n)\geq z+1$. Since $z$ was arbitrary, $v_p(e_n)$ tends to infinity.

If now $K$ is not algebraically closed, it is enough to note that the convergence of $E$ in the $(q(X))$-topology is equivalent to the convergence in $\overline{K}[X]$ of $E$ in the $(X-\lambda)$-topology for every root $\lambda$ of $q(X)$, and then apply the previous reasoning.
\begin{comment}
Suppose first that $K=\overline{K}$ is algebraically closed, let $q(X)=X-\lambda$ and let $Q:=(X-\lambda)$. 

Suppose the two conditions hold, and let $k$ be any integer. There is an $N$ such that $v_p(e_n)\geq p^k$ and $s_n(\lambda)=s$ for every $n\geq N$. Writing $r(X)=r(X-\lambda+\lambda)$, we see that
\begin{equation*}
s_n(X)=u_nr(X)^{p^ke'_n}=u_nr(\lambda)^{p^k}+(X-\lambda)^{p^k}a(X)
\end{equation*}
for some polynomial $a(X)$, and thus $s_n(X)=s_n(\lambda)+(X-\lambda)^{p^k}a(X)\in s+Q^{p^k}$. Since $\{s+Q^{p^k}\}$ is a local basis of neighborhoods of $s$ in the $Q$-topology, $E$ tends to $s$.

Conversely, suppose $E$ converges to $s$ in the $Q$-topology. Then, $s_n(X)\in s+Q$ definitively, i.e., $s_n(X)-s\in Q$, or equivalently $q(X)$ divides $s_n(X)-s$. Hence, $s_n(\lambda)-s=0$ and $s_n(\lambda)=s$ definitively. We now have
\begin{equation*}
s_n(X)=u_n(X-\lambda+\lambda)^{e_n}=u_n\sum_i \binom{e_n}{i}\lambda^{n-i}(X-\lambda)^i.
\end{equation*}
Since $E$ converges to $s$, the polynomial $s_n(X)-s$ must belong (definitively) in $Q^k$ for every $k>0$, that is, the coefficients of degree $<k$ in $X-\lambda$ must be equal to $0$. Therefore, for any fixed $t$ and every large $n$, we must have $\binom{e_n}{t}\lambda^{n-t}=0$. Since $\lambda\neq 0$, it must be $\binom{e_n}{t}=0$, and this implies $v_p(e_n)\geq t$. Since $t$ was arbitrary, $v_p(e_n)$ must tend to infinity.

\end{comment}
\end{proof}

Let now $E$ be a basic $r(X)$-sequence. We denote by $\limiti_K(E)$ the set of maximal ideals $Q$ of $K[X]$, different from $(r(X))$, such that $E$ converges to $1$ in the $Q$-topology; furthermore, we denote by $\limiti_K$ the set of natural numbers $n$ such that there is an irreducible polynomial $r(X)$ and a $(r(X))$-sequence $E$ with $|\limiti_K(E)|=n$. These sets are determined by the Golomb topology, in the following sense.
\begin{prop}\label{prop:limitiKE}
Preserve the notation above.
\begin{enumerate}[(a)]
\item\label{prop:limitiKE:hlim} Let $h:G(K[X])\longrightarrow G(K'[X])$ is a homeomorphism such that $h(1)=1$, and let $r'(X)$ be an irreducible polynomial such that $r'(X)$ generates $h((r(X))$. Then, $h(E)$ is a $r'(X)$-sequence and $h(\limiti_K(E))=\limiti_{K'}(h(E))$.
\item\label{prop:limitiKE:Omef} If $G(K[X])$ and $G(K'[X])$ are homeomorphic, then $\limiti_K=\limiti_{K'}$.
\end{enumerate}
\end{prop}
\begin{proof}
\ref{prop:limitiKE:hlim} follows from the fact that a homeomorphism of Golomb spaces is also a homeomorphism between the $Q$-topology and the $Q'$-topology (where $Q':=h(Q^\nz)\cup\{0\}$). \ref{prop:limitiKE:Omef} follows directly from \ref{prop:limitiKE:hlim}.
\end{proof}

To study $\limiti_K$, we introduce another set associated to a $r(X)$-sequence $E$: we denote by $\ell(E)$ the subset of $\overline{K}$ formed by the roots of the irreducible polynomials that generate a prime ideal of $\mathcal{L}(E)$, that is, $\ell(E)$ is the set of all $\lambda$ such that $E$ converges to $1$ in the $(X-\lambda)$-topology of $\overline{K}[X]$. Note that $\ell(E)$ does not depend on the field $K$, i.e., it remains the same also when considering $E$ in $K'[X]$, where $K'$ is an extension of $K$.
\begin{prop}\label{prop:ellE->sgr}
Let $E$ be a basic $X$-sequence. If $1\in\ell(E)$, then $\ell(E)$ is a torsion multiplicative subgroup of $\overline{K}^\nz$.
\end{prop}
\begin{proof}
Let $E=\{s_n(X):=u_nX^{e_n}\}$. If $1\in\ell(E)$, then $1=s_n(1)$ definitively, that is, $1=u_n1^{e_n}=u_n$ definitively; without loss of generality we can suppose that $u_n=1$ for all $n$. By Proposition \ref{prop:convergenza-Xseq} (and noting that the condition $v_p(e_n)\to\infty$ does not depend on $\lambda$) it follows that $\ell(E)$ is the set of all $\lambda$ such that $\lambda^{e_n}=1$ definitively, and it is easy to see that this is a subgroup of $K^\nz$ whose elements are all torsion.
\end{proof}

The previous proposition also has a converse.
\begin{prop}\label{prop:sgr->ellE}
Let $H$ be a torsion multiplicative subgroup of $\overline{K}^\nz$. Then, there is a basic $X$-sequence $E$ with $\ell(E)=H$.
\end{prop}
\begin{proof}
If $H$ is finite, let $f_n:=|H|$ for all $n$. If $H$ is infinite, let $h_1,h_2,\ldots$ be an enumeration of $H$ (note that, since $H$ is torsion, it is contained in the algebraic closure of $\ins{F}_p$ and thus it is countable), and let $f_n$ be the order of the subgroup generated by $h_1,\ldots,h_n$. We claim that the sequence $E=\{s_n(X):=X^{f_np^n}\}_{n\inN}$ satisfies the condition: indeed, $v_p(e_n)=n$ for all $n$, and thus the $p$-adic valuation of the exponents goes to infinity. Furthermore, if $h\in H$ then $s_n(h)=h^{f_np^n}=(h^{f_n})^{p^n}=1^{p^n}=1$ for all sufficiently large $n$. Thus $h\in\ell(E)$ and $H\subseteq\ell(E)$.

On the other hand, suppose $h\notin H$. If its order is infinite, then $h^{f_np^n}\neq 1$ for every $n$ and $h\notin\ell(E)$ by Proposition \ref{prop:convergenza-Xseq}. If the order of $h$ is finite, we claim that it does not divide any $f_n$. Indeed, every finite subgroup of $K^\nz$ is cyclic, and thus if the order of $h$ divides $f_n$ then $h$ would belong to $\langle h_1,\ldots,h_n\rangle$ and thus to $H$, a contradiction. Since no element of $K^\nz$ has order $p$ (or a multiple of $p$), it follows that the order of $H$ does not divide $f_np^n$ for every $n$; thus, again $h^{f_np^n}\neq 1$ and so $h\notin\ell(E)$. The claim is proved.
\end{proof}

In general, we only know that $|\ell(E)|\leq|\limiti_K(E)|$; however, when $K$ is algebraically closed (or even just separably closed) there is a bijective correspondence between $K$ and $\Max(K[X])$, and thus the two sets have the same cardinality. We now can use the previous propositions to determine $\limiti_K$.
\begin{lemma}\label{lemma:cardsgr}
Let $K$ be an algebraically closed field of characteristic $p>0$, and let $n$ be a positive integer. Then, there is a subgroup of $K^\nz$ of cardinality $n$ if and only if $n$ is coprime with $p$.
\end{lemma}
\begin{proof}
If $n$ is coprime with $p$, then there is a $k$ such that $n$ divides $p^k-1$; therefore, the multiplicative group  of $\ins{F}_{p^k}$ contains a subgroup of cardinality $n$. Since $K$ is algebraically closed, it contains $\ins{F}_{p^k}$, and thus $K^\nz$ contains a subgroup of cardinality $n$.

If $n$ is not coprime with $p$, then $p$ divides $n$. Thus, if $K^\nz$ contains a subgroup of cardinality $n$, it contains also a subgroup of cardinality $p$. However, no element of $K^\nz$ has order $p$.
\end{proof}

\begin{prop}\label{prop:limitiKp}
Let $K$ be a separably field of characteristic $p>0$. Then, $\limiti_K=\insN\setminus p\insN^+$.
\end{prop}
\begin{proof}
Suppose first that $K$ is algebraically closed. If $n>0$ is coprime with $p$, then by Lemma \ref{lemma:cardsgr} there is a subgroup of $K^\nz$ of cardinality $n$, and thus by Proposition \ref{prop:sgr->ellE} there is an $X$-sequence $E$ with $|\limiti_K(E)|=n$. Furthermore, the sequence $\{X^k\}_{k\inN}$ does not converge in any $P$-topology (as $v_p(k)$ does not tend to infinity) and thus also $0\in\limiti_K$. Hence, $\insN\setminus p\insN^+\subseteq\limiti_K$.

Conversely, let $E$ be a $(X-\lambda)$-sequence with $(X-\mu)\in\limiti_K(E)$. Let $\psi$ be the map
\begin{equation*}
\begin{aligned}
\psi\colon G(K[X])& \longrightarrow G(K'[X]),\\
f(X) & \longmapsto f((\lambda+\mu)X+\lambda).
\end{aligned}
\end{equation*}
Then, $\psi$ is a ring automorphism of $K[X]$, and thus it is a self-homeomorphism of $G(K[X])$. Furthermore,
\begin{equation*}
\psi(X-\lambda)=(\lambda+\mu)X+\lambda-\lambda=(\lambda+\mu)X
\end{equation*}
and thus $\psi((X+\lambda))=(X)$; on the other hand,
\begin{equation*}
\psi(X+\mu)=h(1)^{-1}((\lambda+\mu)X-\lambda+\mu=(\lambda+\mu)(X-1)
\end{equation*}
and thus $\psi((X+\mu))=(X-1)$. Hence, $\psi(E)$ is a basic $X$-sequence, and $|\limiti_K(E)|=|\limiti_K(\psi(E))|$. By Lemma \ref{lemma:cardsgr}, $|\limiti_K(E)|$ is coprime with $p$, and thus $\limiti_K\subseteq\insN\setminus p\insN^+$. Thus the two sets are equal.

Suppose now that $K$ is separably closed. Then, every irreducible polynomial is either linear or in the form $X^{p^n}-a$ for some $a\in K$ and some $n\geq 1$; hence, every maximal ideal of $K[X]$ is contained in a a single prime ideal of $\overline{K}[X]$. Therefore, a $r(X)$-sequence $E$ in $K[X]$ is a $s(X)$-sequence in $\overline{K}[X]$, where $s(X)$ generates the prime ideal containing $r(X)$. In particular, $|\limiti_K(E)|=|\limiti_{\overline{K}}(E)|=\ell(E)$; therefore, $\limiti_K=\limiti_{\overline{K}}$ and thus $\limiti_K=\insN\setminus p\insN^+$, as claimed.
\end{proof}

\begin{teor}\label{teor:algchiusi-char}
Let $K,K'$ be two separably closed fields of characteristic $p,p'$ (respectively). If $G(K[X])$ and $G(K'[X])$ are homeomorphic, then $p=p'$.
\end{teor}
\begin{proof}
By Corollary \ref{cor:car0p} we can suppose $p,p'>0$. By Proposition \ref{prop:limitiKE}\ref{prop:limitiKE:Omef}, $\limiti_K=\limiti_{K'}$. By Proposition \ref{prop:limitiKp} $\limiti_K=\insN\setminus p\insN^+$ and $\limiti_{K'}=\insN\setminus p'\insN^+$. Hence, $p=p'$.
\end{proof}

\begin{cor}
Let $K,K'$ be algebraically closed fields, and suppose that $G(K[X])\simeq G(K'[X])$. If one of them is uncountable, then $K\simeq K'$.
\end{cor}
\begin{proof}
Since the cardinality of $K[X]$ is the same of $K$, if $G(K[X])\simeq G(K'[X])$ then $K$ and $K'$ have the same cardinality. If one of them has characteristic 0, then by Proposition \ref{prop:caratt0-ptop} so does the other; otherwise, they have the same positive characteristic by Theorem \ref{teor:algchiusi-char}. Since they are uncountable, algebraically closed and of the same characteristic, by \cite[Chapter VI, Theorem 1.12]{hungerford} $K$ and $K'$ are isomorphic, as claimed.
\end{proof}

In the countable case, we need to distinguish fields that have different degree of transcendence over $\overline{\insQ}$ or $\ins{F}_{p^\infty}$. If the characteristic is positive, the following Proposition \ref{prop:charp-algtrasc} will show that we can distinguish $\ins{F}_{p^\infty}$ from the other fields, but it is an open question if, for example, the algebraic closure of $\ins{F}_p(T)$ and the algebraic closure of $\ins{F}_p(T_1,T_2)$ give rise to non-homeomorphic Golomb spaces.

\section{Almost prime elements}\label{sect:G1}
Let $R$ be a Dedekind domain. An element $b\in R$ is said to be \emph{almost prime} if it is irreducible and it is contained in a unique prime ideal; this happens if and only if $bR=P^n$ for some prime ideal $P$ and $n$ is exactly the order of the class of $P$ in the class group.
\begin{defin}
We say that a Dedekind domain $R$ with torsion class group \emph{has the almost Dirichlet property} (or, simply, that $R$ is \emph{almost Dirichlet}) if any coprime coset contains (at least) one almost prime element, that is, if the set of almost prime elements is dense in $G(R)$.
\end{defin}

\begin{oss}\label{oss:dirichlet}
~\begin{enumerate}[(1)]
\item If $R$ has torsion class group, $h:G(R)\longrightarrow G(S)$ is a homeomorphisms of Golomb spaces and $b\in R$ is contained in a unique prime ideal, the same happens for $h(b)$ \cite[Proposition 2.7]{golomb-almcyc}. However, it is an open question whether $h$ sends irreducible elements into irreducible elements; in particular, we do not know if the almost Dedekind property is a topological invariant (with respect to the Golomb topology).
\item If $R$ is almost Dirichlet, then $G_1(R)$ is dense in $G(R)$, as every almost prime element belongs to $G_1(R)$.
\item\label{oss:dirichlet:2} By Dirichlet's theorem on primes in arithmetic progressions, the ring $\insZ$ of integers is almost Dirichlet. The same happens when $R=\insQ[X]$ or $R=F[X]$, where $F$ is a finite field \cite[Theorem 4.8]{rosen}.
\item\label{oss:dirichlet:pac} A field $F$ is said to be \emph{pseudo-algebraically closed} (PAC) if every nonempty absolutely irreducible variety defined over $F$ has an $F$-rational point \cite[Chapter 11]{field_arithmetic}. If $F$ is PAC and contains separable irreducible polynomials of arbitrarily large degree, then every coprime coset contains irreducible polynomials, and $F[X]$ has the almost Dirichlet property \cite[Theorem A]{dirichlet-polynomial}.
\end{enumerate}
\end{oss}

\begin{prop}\label{prop:dirichlet-aeff}
Let $F$ be an algebraic extension of a finite field that is not algebraically closed. Then, $F[X]$ has the almost Dirichlet property.
\end{prop}
\begin{proof}
If $F$ is finite, the claim follows from \cite[Theorem 4.8]{rosen}. If not, then $F$ is pseudo-algebraically closed \cite[Corollary 11.2.4]{field_arithmetic} and has simple separable extensions of arbitrarily large degree, and thus $F[X]$ is almost Dirichlet by \cite[Theorem A]{dirichlet-polynomial}.
\end{proof}

A simple consequence of the Remark \ref{oss:dirichlet}\ref{oss:dirichlet:2} and of Proposition \ref{prop:C-GnK} is the following.
\begin{cor}
$G(\insQ[X])\not\simeq G(\overline{\insQ}[X])$.
\end{cor}

We now want to prove that, at least in some cases, a homeomorphism of Golomb spaces preserves almost prime elements and, to do so, we shall abstract the proof of \cite[Lemmas 5.10 and 5.11]{bmt-golomb}.	
\begin{defin}
Let $R$ be a Dedekind domain with torsion class group. We say that $R$ is \emph{power separated} if, for every maximal ideal $P$ and every $b\in G_{\{P\}}(R)$, we have $\overline{\pow{b}}\cap G_{\{P\}}(R)=\pow{b}$.
\end{defin}

A more explicit sufficient condition is the following.
\begin{prop}
Let $R$ be a Dedekind domain, and suppose there is a function $d:R^\nz\longrightarrow[1,+\infty)$ such that, for all $a,b\in R^\nz$:
\begin{itemize}
\item $d(ab)=d(a)d(b)$;
\item $d(a+b)\leq d(a)+d(b)$ if $a\neq-b$;
\item $d(a)=1$ if and only if $a$ is a unit.
\end{itemize}
Then, $R$ is power separated.
\end{prop}
\begin{proof}
Let $P$ be a prime ideal, $b\in G_{\{P\}}(R)$ and $c\in G_{\{P\}}(R)\setminus\pow{b}$. By hypothesis, $d(b)>1$, and thus we can find an integer $t$ such that $d(b)^t>d(b)^{t-1}+d(c)+1$. Let $I:=(b^t-1)R$: then, $c+I$ is open (since $b^t-1\notin P$), and we claim that $(c+I)\cap\pow{b}=\emptyset$.

Indeed, suppose not, and let $z$ be in the intersection. Then, $z=ub^r$ for some $u\in U(R)$, $r\inN$. Since $b^t\equiv 1\bmod I$, we see that $z\equiv ub^s\bmod I$ for some $s\in\{0,\ldots,t-1\}$ (setting $b^0:=1$), and thus $c\equiv ub^s\bmod I$, i.e., $c-ub^s\in I$. However, as $c\neq ub^s$ we can calculate
\begin{equation*}
d(c-ub^s)\leq d(c)+d(ub^s)=d(c)+d(b)^s\leq d(c)+d(b)^{t-1}<d(b)^t-1\leq d(b^t-1).
\end{equation*}
For all $x\in I$, we have $d(x)\geq d(b^t-1)$; this is a contradiction, and thus $c+(b^t-1)R$ does not meet $\pow{b}$. Therefore, $\pow{b}$ is closed in $G_{\{P\}}(R)$, and thus $R$ is power separated.
\end{proof}

\begin{cor}
The following hold.
\begin{enumerate}[(a)]
\item If $R$ is contained in $\insC$ and has only finitely many units, then $R$ is power separated.
\item If $R=K[X]$ for some field $K$, then $R$ is power separated.
\end{enumerate}
\end{cor}
\begin{proof}
In the first case, all units of $R$ are roots of unity, so we can take the complex modulus as $d$. For the second case, set $d(p):=2^{\deg(p)}$.
\end{proof}

\begin{comment}
A more explicit sufficient condition is the following.
\begin{prop}
Let $R$ be a Dedekind domain with torsion class group. Suppose that, for every maximal ideal $P$ and every $b,c\in G_{\{P\}}(R)$ with $c\notin\pow{b}$, there is a $t$ such that $ub^s\not\equiv c\bmod(b^t-1)$ for every $s\in\{0,\ldots,t-1\}$. Then, $R$ is power separated.
\end{prop}
\begin{proof}
Let $I=(b^t-1)$; we claim that $(c+I)\cap\pow{b}=\emptyset$. Indeed, suppose not: then, there are $i\in I$, $u\in U(R)$ and $r\inN$ such that $c+i=ub^s$, and thus $c\equiv ub^r\bmod(b^t-1)$. Write $r:=tq+s$ with $0\leq s<t$: then, $ub^r=ub^{tq+s}\equiv ub^s\bmod (b^t-1)$. Hence, $c\equiv ub^s\bmod(b^t-1)$, against the hypothesis.
\end{proof}

\begin{cor}
The following hold.
\begin{enumerate}
\item If $R$ is contained in $\insC$ and has only finitely many units, then $R$ is power separated.
\item If $R=K[X]$ for some field $K$, then $R$ is power separated.
\end{enumerate}
\end{cor}
\begin{proof}
In the first case, all units of $R$ are roots of unity, so they have modulus $1$. Hence,
\begin{equation*}
|ub^s-c|\leq|ub^s|+|c|\leq|b|^{t-1}+|c|<|b|^t-1\leq|b^t-1|
\end{equation*}
for sufficiently large $t$, and thus $b^t-1$ cannot divide $ub^s-c$ unless the latter is $0$, i.e., $c\not\equiv ub^s\bmod(b^t-1)$ unless $c\in\pow{b}$.

For the second case, take $t>\deg c$: then, $\deg(ub^s-c)<\deg(b^t-1)$ for $s<t$, and as before $b^t-1$ cannot divide $ub^s-c$ unless the latter is $0$.
\end{proof}
\end{comment}

%The following is an abstraction the proof of \cite[Lemma 5.10]{bmt-golomb}.
\begin{teor}\label{teor:almprime}
Let $R$ be a Dedekind domain with torsion class group, and suppose that $R$ is power separated and has the almost Dirichlet property. If $S$ is a Dedekind domain and $h:G(S)\longrightarrow G(R)$ is a homeomorphism, then $h$ sends almost prime elements into almost prime elements.
\end{teor}
\begin{proof}
Let $a\in S$ be an element contained in a unique prime ideal, and let $b:=h(a)$. We first claim that $h(\pow{a})\subseteq\pow{b}$.

Fix a unit $u_0\in S$ and an integer $n\geq 1$. Let $f:G(S)\longrightarrow G(S)$ be the map sending every $x$ to $u_0x^n$, and let $\phi:G(R)\longrightarrow G(R)$ be the composition $h\circ f\circ h^{-1}$. Then, $f$ is continuous in the Golomb topology, and thus so is $\phi$; furthermore, if $P$ is a prime ideal of $R$, then $h\circ f\circ h^{-1}(P)\subseteq P$ since $h^{-1}(P)$ is a prime ideal of $S$. Let
\begin{equation*}
c:=\phi(b)=\phi(h(a))=(h\circ f\circ h^{-1}\circ h)(a)=h(u_0a^n).
\end{equation*}
Suppose that $c\notin\pow{b}$: then, since $R$ is power separated, we can find an open set $\Omega:=c+I$ such that $\Omega\cap\pow{b}$ does not meet $G_{\{Q\}}(R)$ (where $Q$ is the radical of $bR$). Since $\phi$ is continuous, $\phi^{-1}(\Omega)$ is an open set containing $b$; hence, there is a $d\in R$, coprime with $b$, such that $\phi(b+dR)\subseteq\Omega$. 

Since $R$ is almost Dirichlet we can find an almost prime element $p\in b+dI$. Then, $\pow{p}=G_{\{P\}}(R)$, where $P$ is the only prime ideal containing $p$; hence, $\phi(p)\in\pow{p}$, i.e., there are $u\in U(R)$ and $l\inN^+$ such that $\phi(p)=up^l$. On the other hand,
\begin{equation*}
\phi(p)\in\phi(b+dR)\subseteq\Omega=c+I
\end{equation*}
and, at the same time,
\begin{equation*}
up^l\in u(b+dI)^l\subseteq u(b^l+I)=ub^l+I;
\end{equation*}
it follows that $c\equiv ub^l\bmod I$, i.e., $ub^l\in c+I=\Omega$. This contradicts the choice of $I$; hence, $c$ must be in $\pow{b}$, that is, $h(u_0a^n)=c=ub^l$ for some $l$. Since this happens for every $u_0$ and every $n$, we have $h(\pow{a})\subseteq\pow{b}$.

Suppose now that $a$ is almost prime, and let $P$ and $Q$ be, respectively, the only prime ideal containing $a$ and the only prime ideal containing $b$. Then,
\begin{equation*}
G_{\{Q\}}(R)=h(G_{\{P\}}(S))=h(\pow{a})\subseteq\pow{b}\subseteq G_{\{Q\}}(R).
\end{equation*}
Thus $\pow{b}=G_{\{Q\}}(R)$, i.e., $b$ is almost prime.
\end{proof}

\section{Algebraic extensions of $\ins{F}_p$}\label{sect:algFp}
As observed in \cite[Corollary 14]{clark-golomb}, a consequence of the fact that a homeomorphism of Golomb spaces sends units to units is that if $K,K'$ are distinct finite fields then the Golomb spaces $G(K[X])$ and $G(K'[X])$ are not homeomorphic. The purpose of this section is to generalize this result, allowing $K$ and $K'$ to be arbitrary algebraic extensions of the same $\ins{F}_p$.

The first step is to distinguish algebraic extensions from transcendental extensions. 
\begin{prop}\label{prop:charp-algtrasc}
Let $K$ be a field of characteristic $p>0$ and let $g\in K[X]$ be an irreducible polynomial. Then, the following are equivalent.
\begin{enumerate}[(i)]
\item\label{prop:charp-algtrasc:disc} $\pow{g}$ is not discrete in $G(K[X])$;
\item\label{prop:charp-algtrasc:s1s2} for every $s_1,s_2\in K$, either $g(s_1)=0$, $g(s_2)=0$ or $g(s_1)/g(s_2)$ is a root of unity;
\item\label{prop:charp-algtrasc:alg} $K$ is algebraic over $\ins{F}_p$.
\end{enumerate}
\end{prop}
\begin{proof}
\ref{prop:charp-algtrasc:disc} $\Longrightarrow$ \ref{prop:charp-algtrasc:s1s2} Fix $\lambda:=ug^n\in\pow{g}$. Let $s_1,s_2\in K$ be such that $g(s_1)\neq 0\neq g(s_2)$, and let $I$ be the ideal of $K[X]$ generated by $(X-s_1)(X-s_2)$: then, $I$ is coprime with $g$, and thus $\lambda+I$ is an open subset of $G(K[X])$. Since $\pow{g}$ is not discrete, there are infinitely many $\lambda':=u'g^m\in\lambda+I$, with $\lambda'\neq\lambda$.

Therefore, $I$ contains $u'g^m-ug^n=u'g^n(g^r-v)$, where $r:=m-n$ and $v:=uu'^{-1}$ (with $g^0:=1$); setting $h:=g^r-v$, it follows that $h(s_1)=h(s_2)=0$, and thus that $r>0$ (since if $r=0$ then $v\neq 1$ and $h$ is a nonzero constant) and $g(s_1)^r=v=g(s_2)^r$. Hence, $(g(s_1)/g(s_2))^r=v/v=1$; that is, $g(s_1)/g(s_2)$ is a root of unity, as claimed.

\ref{prop:charp-algtrasc:s1s2} $\Longrightarrow$ \ref{prop:charp-algtrasc:alg} Suppose not: then, $K$ is infinite. Let $s_1$ be any element of $K$ such that $g(s_1)\neq 0$. Let $F$ be field generated by $s_1$, the coefficients of $g$ and an element of $K$ that is transcendental over the prime field: then, $F$ is infinite and contains only finitely many roots of unity. Hence, there are only finitely many $t\in F$ such that $g(t)=0$ or $g(t)=ug(s_1)$ for some root of unity $u$ in $F$. In particular, there is an $s_2$ which does not satisfy either equality; however, this contradicts the hypothesis, and thus $K$ is algebraic over $\ins{F}_p$.

\ref{prop:charp-algtrasc:alg} $\Longrightarrow$ \ref{prop:charp-algtrasc:disc} Let $\lambda\in\pow{g}$, and let $I$ be an ideal of $K[X]$ that is coprime with $g$ (and thus with $\lambda$); let $f$ be a generator of $I$. We need to show that the open set $\lambda+I$ contains other elements of $\pow{g}$.

Let $F$ be the subfield of $K$ generated by $u$, the coefficients of $g$ and by the roots of $f$: then, $F$ is a finite field, say of cardinality $q$. For every $\alpha\in F$, $\lambda(\alpha)^{q-1}=1$; hence, the polynomial $h:=1-\lambda^{q-1}$ has zeros in every element of $F$, and in particular all the zeros of $f$ are zeros of $\lambda'$. Let $q'$ be a power of $q$ greater than every multiplicity of the roots of $f$: then, $f$ divides $h^{q'}=(1-\lambda^{q-1})^{q'}=1-\lambda^{q'(q-1)}$. Therefore,
\begin{equation*}
\lambda-\lambda^{q'(q-1)+1}=\lambda(1-\lambda^{q'(q-1)})\in I,
\end{equation*}
and thus $\lambda^{q'(q-1)+1}\in\lambda+I$, as claimed.
\end{proof}

\begin{cor}\label{cor:charp-alg}
Let $K_1,K_2$ be two field of positive characteristic. If $K_1$ is algebraic over its base field while $K_2$ is not then $G(K_1[X])\not\simeq G(K_2[X])$.
\end{cor}

Let $\omef(G(R))$ be the group of self-homeomorphisms of $G(R)$, and let
\begin{equation*}
\Lambda(R):=\{h\in\omef(G(R))\mid h(P^\nz)=P^\nz\text{~for every~}P\in\Spec(R)\}
\end{equation*}
and
\begin{equation*}
\Lambda_1(R):=\{h\in\Lambda(R)\mid h(1)=1\}.
\end{equation*}
Note that $\Lambda(R)$ does not necessarily contain all self-homeomorphisms of $G(R)$: for example, a ring automorphism $\psi$ of $R$ induces a self-homeomorphism of $\Lambda(R)$, but usually does not fix all prime ideals. (For an example, take $R=\insZ[i]$ and let $\psi$ be the complex conjugation.)

These groups are effectively invariants of the Golomb topology.
\begin{prop}
Let $R,S$ be two Dedekind domains, and suppose $G(R)$ and $G(S)$ are homeomorphic. Then, $\Lambda(R)\simeq\Lambda(S)$ and $\Lambda_1(R)\simeq\Lambda_1(S)$.
\end{prop}
\begin{proof}
Let $h:G(R)\longrightarrow G(S)$ be a homeomorphism. For every $\psi\in\Lambda(R)$, the map $\overline{\psi}:=h\circ\psi\circ h^{-1}$ is a self-homeomorphism of $G(S)$, and if $P$ is a prime ideal of $R$ then $\overline{\psi}(P^\nz)=h(\psi(h^{-1}(P^\nz)))=h(h^{-1}(P^\nz))=P^\nz$; thus, $\overline{\psi}\in\Lambda(S)$. Hence, $h$ induces a map $\Lambda(R)\longrightarrow\Lambda(S)$, sending $\psi$ to $\overline{\psi}$, which is easily seen to be a group homomorphism. Likewise, $h^{-1}$ induces a map $\Lambda(S)\longrightarrow\Lambda(R)$ which is the inverse of the previous one. Hence, $\Lambda(R)\simeq\Lambda(S)$.

The reasoning for $\Lambda_1$ is the same, using the homeomorphism $h':G(R)\longrightarrow G(S)$ sending $x$ to $h(1)^{-1}h(x)$ (so that $h'(1)=1$).
\end{proof}

For any unit $u$ of $R$, let $\psi_u$ be the multiplication by $u$, and let $H:=\{\psi_u\mid u\in U(R)\}$. Then, $H$ is a subgroup of $\Lambda(R)$ (and thus of $\omef(R)$) that is isomorphic to the group of units of $R$. For every $h\in\omef(G(R))$, the map $h_1:=\psi_{h(1)^{-1}}\circ h$ is a self-homeomorphism of $G(R)$ fixing $1$; furthermore, if $h\in\Lambda(R)$ then so does $h_1$, and thus $h_1\in\Lambda_1(R)$. It follows that $\Lambda(R)$ is generated by $H$ and $\Lambda_1(R)$, and in particular if $\Lambda_1(R)$ is trivial then $\Lambda(R)=H\simeq U(R)$.

For example, if $R=\insZ$ then by \cite[Theorem 6.7]{golomb-almcyc} $\Lambda_1(\insZ)$ is trivial and thus $\Lambda(\insZ)$ is isomorphic to $U(\insZ)\simeq\insZ_2$. This phenomenon is linked to the hypothesis we worked with in Section \ref{sect:G1}.
\begin{prop}\label{prop:Homeofix}
Let $R$ be a Dedekind domain with torsion class group that has the almost Dirichlet property and is power separated. Suppose that there are infinitely many prime ideals $P$ such that $U(R)\longrightarrow R/P^{n_P}$ is injective for some integer $n_P$. Then, $\Lambda_1(R)$ is trivial and $\Lambda(R)\simeq U(R)$.
\end{prop}
\begin{proof}
Let $\Delta$ be the set of all prime ideals for which there is such a $n_P$, and let $X:=\bigcup\{P^\nz\mid P\in\Delta\}$.

By Theorem \ref{teor:almprime}, any self-homeomorphism $h$ of $G(R)$ sends almost prime elements into almost prime elements. Let $h\in\Lambda_1(R)$, and let $f$ be almost prime: then, $h(f)$ is an almost prime element contained in the same prime ideal of $f$, and thus there is a $u_f\in U(R)$ such that $h(f)=u_ff$.

Let $P\in\Delta$. Then, $h$ is a homeomorphism in the $P$-topology, and thus in particular it is continuous, i.e., for every $n$ there is an $m=m(n)\geq n$ such that $h(1+P^m)\subseteq 1+P^n$ (using $h(1)=1$). Choose $n\geq n_P$: then, for every $f\in 1+P^m$ that is almost prime both $f$ and $u_ff$ are in $1+P^n$, and thus $f-u_ff=f(1-u_f)\in P^n$. Since $f\notin P$, it follows that $1-u_f\in P^n$. By the injectivity of $U(R)\longrightarrow R/P^n$ we have $u_f=1$, i.e., $f$ is a fixed point of $h$. The closure of $1+P^m$ is the Golomb topology is $(1+P^m)\cup P^\nz$; hence, also all the elements of $P^\nz$ are fixed points of $h$. It follows that $h|_X$ is the identity.

Let now $z\in G(R)$ and let $z+I$ be an open neighborhood of $z$. Since $\Delta$ is infinite, there is a $Q\in\Delta$ that is coprime with $I$ and $z$; thus, $z+I$ meets $Q$. Since $I$ was arbitrary, it follows that $z$ is in the closure of $X$; thus, $X$ is dense in $G(R)$. Since $h|_X$ is the identity, the whole $h$ is the identity. Hence, $\Lambda_1(R)$ is trivial and $\Lambda(R)\simeq U(R)$.
\end{proof}

\begin{teor}\label{teor:algextFp}
Let $K,K'$ be fields of characteristic $p>0$. If $K$ is algebraic over $\ins{F}_p$ and $G(K[X])\simeq G(K'[X])$ then $K\simeq K'$.
\end{teor}
\begin{proof}
By Corollary \ref{cor:charp-alg}, $K'$ must be algebraic over $\ins{F}_p$. If $K$ is algebraically closed, then $G_1(K[X])$ is not dense in $G(K[X])$ (Proposition \ref{prop:Gn-pinfty}\ref{prop:Gn-pinfty:dense}); if $K'$ is not algebraically closed, then $K'[X]$ is almost Dirichlet (Proposition \ref{prop:dirichlet-aeff}) and thus $G_1(K[X])$ is dense in $G(K[X])$. Therefore, if $K$ is algebraically closed then so is $K'$, and thus $K\simeq K'$.

Suppose now that $K$ is not algebraically closed. By the previous reasoning, neither $K'$ is algebraically closed. By Proposition \ref{prop:dirichlet-aeff}, $K[X]$ and $K'[X]$ are almost Dirichlet, and thus by Proposition \ref{prop:Homeofix} $\Lambda_1(K[X])\simeq U(K[X])=K^\nz$ and $\Lambda_1(K'[X])\simeq U(K'[X])=K'^\nz$. Furthermore, all maps $K^\nz\longrightarrow K[X]/P$ are injective; by Proposition \ref{prop:Homeofix}, it follows that $K^\nz\simeq K'^\nz$.

We can consider $K$ and $K'$ contained in the algebraic closure $\ins{F}_{p^\infty}$. If $K'$ is not isomorphic to $K$, then $K\neq K'$, and thus without loss of generality there is a finite extension $\ins{F}_{p^n}$ that is contained in $K$ but not in $K'$. Hence, $K^\nz$ contains elements of order $p^n-1$ (the generator of the multiplicative group of $\ins{F}_{p^n}$) while $K'$ does not, because $p^m-1$ is a multiple of $p^n-1$ only if $m$ is a multiple of $n$. Therefore, $K^\nz\simeq K'^\nz$ implies $K=K'$, as claimed.
\end{proof}

As a corollary, we are able to answer affirmatively to a question posed in \cite[Section 3.1]{clark-golomb}. We denote by $\mathfrak{c}$ the cardinality of the continuum.
\begin{cor}
The number of distinct Golomb topologies associated to countably infinite domains is $\mathfrak{c}$.
\end{cor}
\begin{proof}
There are at most $\mathfrak{c}$ possible binary operations on a countably infinite set, and thus there are at most $\mathfrak{c}$ distinct Golomb topologies.

To show that there are exactly $\mathfrak{c}$, let $p$ be a prime number and let $\mathcal{C}_p$ be the set of all (isomorphism classes of) algebraic extensions of $\ins{F}_p$. By Theorem \ref{teor:algextFp}, the Golomb topologies relative to the members of $\mathcal{C}_p$ are pairwise non-homeomorphic, and thus we need to show that $\mathcal{C}_p$ has cardinality at least $\mathfrak{c}$.

Let $\{q_1,q_2,\ldots\}$ be the set of prime numbers. To each $A\subseteq\insN$, we can associate the field $F(A)$ defined as the composition of the extensions of $\ins{F}_p$ of degree $q_i$, for $i\in A$: then, $F(A)\neq F(A')$ if $A\neq A'$, and thus the cardinality of $\mathcal{C}_p$ is at least the cardinality of the power set of $\insN$, i.e., $\mathfrak{c}$. The claim is proved.
%Let $\{q_1,q_2,\ldots\}$ be the set of prime numbers. For each $K\in\mathcal{C}_p$ and each $i$, let $n_i(K):=\sup \{n\mid K$ contains the extension of degree $q_i^{n_i}\}\inN\cup\{\infty\}$. Then, to each $K$ is associated the sequence $n(K)=(n_1(K),n_2(K),\ldots)$  and, conversely, to each sequence $n:=(n_1,n_2,\ldots)$ we can associate a field $K(n)$ taking the compositum of the extensions of degree $q_i^{n_i}$; going back, $n(K(n))$ will be the original sequence. Thus, we have a correspondence between $\mathcal{C}_p$ and the set of ordered sequences in $\insN\cup\{\infty\}$; since there are $\mathfrak{c}$ many of these also $\mathcal{C}_p$ has cardinality $\mathfrak{c}$, as claimed.
\end{proof}

The method used in the proof of Theorem \ref{teor:algextFp} does not quite extend to the case in which the characteristic of $K$ and $K'$ are not supposed beforehand to be equal; that is, it is not clear how to prove the analogous of Theorem \ref{teor:algchiusi-char} for algebraic extensions of finite fields. We can however say something about the relation between the two characteristics.
\begin{prop}
Let $K,K'$ be algebraic extensions of $\ins{F}_p$ and $\ins{F}_{p'}$, respectively. If $p$ divides $p'-1$, then $G(K[X])$ and $G(K'[X])$ are not homeomorphic.
\end{prop}
\begin{proof}
Using Theorem \ref{teor:algchiusi-char} we can suppose that $K$ and $K'$ are not algebraically closed. As in the proof of Theorem \ref{teor:algextFp}, by Propositions \ref{prop:dirichlet-aeff} and \ref{prop:Homeofix} if $G(K[X])\simeq G(K'[X])$ then the groups of units $K^\nz$ and $K'^\nz$ are isomorphic. However, $p|p'-1$ implies that there is an $u\in K'^\nz$ of order $p$, something which cannot happen in $K^\nz$. Hence, $G(K[X])$ and $G(K'[X])$ are not homeomorphic.
\end{proof}

\bibliographystyle{plain}
\bibliography{/bib/articoli,/bib/libri,/bib/miei}
\end{document}